\documentclass[12pt]{amsart}

\usepackage{amssymb}
\usepackage{bm}
\usepackage{graphicx}
\usepackage{psfrag}
\usepackage{color}
\usepackage{url}

\usepackage{algpseudocode}

\usepackage{mathtools}

\usepackage{xy}
\input xy
\xyoption{all}

\newcommand{\excise}[1]{}

\newtheorem{thm}{Theorem}[section]
\newtheorem{lemma}[thm]{Lemma}

\newtheorem{cor}[thm]{Corollary}
\newtheorem{prop}[thm]{Proposition}

\newtheorem{prob}[thm]{Problem}

\theoremstyle{definition}
\newtheorem{alg}[thm]{Algorithm}
\newtheorem{example}[thm]{Example}
\newtheorem{remark}[thm]{Remark}
\newtheorem{defn}[thm]{Definition}

\numberwithin{equation}{section}

\newcommand{\ring}[1]{\ensuremath{\mathbb{#1}}}

\renewcommand\>{\rangle}
\newcommand\<{\langle}

\newcommand\NN{\ring{N}}

\newcommand\QQ{\ring{Q}}

\newcommand\ZZ{\ring{Z}}

\renewcommand\aa{{\mathbf a}}
\newcommand\bb{{\mathbf b}}
\newcommand\ee{{\mathbf e}}

\newcommand\qq{{\mathsf q}}

\newcommand\iso{\cong}

\renewcommand\implies{\Rightarrow}

\DeclareMathOperator\lcm{lcm} 
\DeclareMathOperator\Ap{Ap} 
\DeclareMathOperator\bul{bul} 

\begin{document}

\mbox{}
\title{On dynamic algorithms for factorization invariants in numerical monoids}
\author{Thomas Barron}
\address{Mathematics Department\\University of Kentucky\\Lexington, KY 40506}
\email{thomas.barron@uky.edu}
\author{Christopher O'Neill}
\address{Mathematics Department\\Texas A\&M University\\College Station, TX 77843}
\email{coneill@math.tamu.edu}
\author{Roberto Pelayo}
\address{Mathematics Department\\University of Hawai`i at Hilo\\Hilo, HI 96720}
\email{robertop@hawaii.edu}

\date{\today}

\begin{abstract}
\hspace{-2.05032pt}
Studying the factorization theory of numerical monoids relies on understanding several important factorization invariants, including length sets, delta sets, and $\omega$-primality.  While progress in this field has been accelerated by the use of computer algebra systems, many existing algorithms are computationally infeasible for numerical monoids with several irreducible elements.  In this paper, we present dynamic algorithms for the factorization set, length set, delta set, and $\omega$-primality in numerical monoids and demonstrate that these algorithms give significant improvements in runtime and memory usage.  In describing our dynamic approach to computing $\omega$-primality, we extend the usual definition of this invariant to the quotient group of the monoid and show that several useful results naturally extend to this broader setting. 
\end{abstract}
\maketitle


\section{Introduction}\label{s:intro}

Numerical monoids (co-finite, additive submonoids of $\NN$) and their non-unique factorization theory have emerged as an active area of study in recent years~\cite{andalg,elastsets,delta,deltaperiodic,omegaquasi,omegamonthly}.  The primary objects in these investigations are factorization invariants, which measure in different ways the plural factorizations of an element $n$ into irreducibles in the numerical monoid $S$.  The set $\mathsf Z(n)$ of factorizations of a single element $n \in S$ can uniquely determine the monoid structure of $S$.  As such, $\mathsf Z(n)$ is often too cumbersome to compute for large values of $n$ in numerical monoids with several irreducible elements.  Most of the major factorization invariants of interest, including the length set $\mathsf L(n)$, delta set $\Delta(n)$, $\omega$-primality $\omega(n)$, and catenary degree $\mathsf c(n)$, are typically only computed after first computing $\mathsf Z(n)$.  

Discovering many of the results regarding these factorization invariants has relied on effectively utilizing computer algebra packages, such as the GAP package \texttt{numericalsgps}~\cite{numericalsgps}.  Of particular interest is finding closed forms for and analyzing the asymptotic behavior of these invariants, which frequently requires computing the invariant for numerous elements of $S$.   As each of these individual computations generally requires first computing $\mathsf Z(n)$ and then passing to the invariant of interest, finding patterns for numerical monoids with large numbers of irreducible elements is often computationally infeasible.

Recent investigations in numerical monoids show that their factorization invariants tend to have eventual quasi-polynomial behavior.  For example, $\omega$-primality~\cite{omegaquasi} and maximal and minimal elements of length sets~\cite{elastsets} have quasilinear behavior, while delta sets~\cite{deltaperiodic} and catenary degree~\cite{catenaryperiodic} have periodic (i.e., quasi-constant) behavior.  The inductive nature of these results motivates the major results for this paper: several factorization invariants, including $\mathsf Z(n), \mathsf L(n), \Delta(n),$ and $\mathsf \omega(n)$, can be computed dynamically.   Moreover, the dynamic algorithms we present for $\mathsf L(n), \Delta(n),$ and $\omega(n)$ (Algorithms~\ref{a:dynamiclengths},~\ref{a:dynamicdelta}, and~\ref{a:dynamicomega}, respectively) do not require computing $\mathsf Z(n)$, which significantly improves runtimes and memory usage.  Additionally, as closed forms and asymptotic behavior is of interest, a dynamic algorithm, which computes several values quickly in succession, is a fruitful approach.
  
In Section~\ref{s:background}, we provide all necessary definitions, including those for numerical monoids (Definition~\ref{d:numericalmonoid}) and their length set and delta set invariants (Definition~\ref{d:deltaset}), and state a recent periodicity result of interest.  In Section~\ref{s:dynamicfactor}, we give a dynamic algorithm to compute the delta set of a numerical monoid (Algorithm~\ref{a:dynamicdelta}), and compare this algorithm to existing algorithms~\cite{compasympdelta}.  In Section~\ref{s:dynamicomega}, we give a dynamic algorithm to compute $\omega$-primality in numerical monoids (Algorithm~\ref{a:dynamicomega}), after developing the necessary theory in Section~\ref{s:omegaext}.  In Section~\ref{s:omegabound}, we prove Theorem~\ref{t:omegafactor}, which relates factorizations and bullets (Definition~\ref{d:bulletext}), and derive an improved bound on start of quasilinear behavior of the $\omega$-primality function (Theorem~\ref{t:newomegabound}).  Finally, Section~\ref{s:futurework} states several open questions.  

At the time of writing, Algorithms~\ref{a:dynamicdelta}, \ref{a:dynamicfactor}, and~\ref{a:dynamicomega} are currently implemented in the GAP package \texttt{numericalsgps}~\cite{numericalsgps}.  All included benchmarks use this software, including comparisons to existing algorithms.  

\subsection*{Acknowledgements}
The authors would like to thank Pedro Garc\'ia-S\'anchez for numerous insightful conversations.

\section{Background}\label{s:background}

We begin by defining numerical monoids (Definition~\ref{d:numericalmonoid}) and introducing three of the main factorization invariants of interest: factorization sets, length sets, and delta sets (Definitions~\ref{d:factorization} and~\ref{d:deltaset}).  In this section, and in the majority of this paper, monoids are written additively, and $\mathbb N$ denotes the set of non-negative integers.

\begin{defn}\label{d:monoid}
A commutative monoid $M$ (written additively) is \emph{cancellative} if for any $a,b,c \in M$, $a + b = a + c$ implies that $b = c$.  An element $u \in M$ is said to be \emph{irreducible} (or an \emph{atom}) if whenever $u = a + b$ for $a,b \in M$, then either $a$ or $b$ is a unit in $M$.  The monoid $M$ is said to be \emph{atomic} if for every $m \in M$, there exist irreducible elements $u_1, u_2, \ldots, u_r \in M$ such that $m = u_1 + u_2 +  \cdots + u_r$.  
\end{defn}

\begin{remark}\label{r:monoid}
Unless otherwise stated, all monoids in this paper are assumed to be cancellative, commutative, and atomic.  
\end{remark}

This paper will focus on the factorization theory of additive submonoids of $\NN$.

\begin{defn}\label{d:numericalmonoid}
For relatively prime positive integers $n_1, n_2, \ldots, n_k \in \NN$, the \emph{numerical monoid $S$ generated by $\{n_1, n_2, \ldots, n_k\}$} is the additive submonoid of $\NN$ given by
$$S = \<n_1, n_2, \ldots, n_k\> = \{c_1n_1 + c_2n_2 + \cdots + c_kn_k  : c_i \in \mathbb N\} \subset \NN.$$
The \emph{Frobenius number} $F(S) = \max(\NN \setminus S)$ is the largest integer lying outside of $S$, and an element $n \in \ZZ \setminus S$ is a \emph{pseudo-Frobenius number} if $n + n_i \in S$ for all $n_i$.  
\end{defn}

\begin{remark}\label{r:numericalmonoid}
Every co-finite, additive submonoid of $\NN$ is a numerical monoid generated by a set of relatively prime generators $\{n_1, n_2, \ldots, n_k\}$.  In fact, there exists a unique collection of generators that is minimal with respect to set-theoretic inclusion.  For this generating set, the irreducible elements of the numerical monoid coincide precisely with the generators.  Thus, every numerical monoid is cancellative, reduced, and finitely generated. 
Unless otherwise stated, when we write $S = \<n_1, n_2, \ldots, n_k \>$, we assume that the $n_i$ constitute a minimal generating set and that $n_1 < n_2 < \cdots < n_k$.  
\end{remark}

Of particular interest in the study of numerical monoids is its non-unique factorization theory.  We establish here notation for factorizations of elements in these monoids.

\begin{defn}\label{d:factorization}
Let $S = \<n_1, n_2, \ldots, n_k\>$ be a numerical monoid.  A \emph{factorization} of $x \in S$ is a $k$-tuple $\aa = (a_1, a_2, \ldots, a_k) \in \NN^k$ such that $x = a_1n_1 + a_2n_2 + \cdots + a_kn_k$, and the \emph{set of factorizations of $x$ in $S$} is given by $$\mathsf Z_S(x) = \{\aa \in \NN ^k \, : \, a_1n_1 + a_2n_2 + \cdots + a_kn_k = x \}.$$  When the monoid is understood, this is often denoted $\mathsf Z(x)$.  The \emph{length} of a factorization  $\aa$ is given by $|\aa| = \sum_{i=1}^k a_i$. 
\end{defn}

In a numerical monoid $S$, the set of factorizations $\mathsf Z(x)$ is usually cumbersome, especially for larger values of $x \in S$.  Therefore, many factorization invariants are computed using only the lengths of factorizations.  This leads us to the concepts of the length set and delta set.

\begin{defn}\label{d:deltaset}  Given a numerical monoid $S = \<n_1, n_2, \ldots, n_k \>$ and $x \in S$, the \emph{length set of $x$} is given by $$\mathsf L(x) = \{ |\aa| \, : \, \aa \in \mathsf Z(x)\}.$$  If we order the elements of $\mathsf L(x)$ in increasing order $\mathsf L(x) = \{l_1 < l_2 < \cdots < l_r\}$, the \emph{delta set of $x$} is given by $$\Delta(x) = \{l_i - l_{i-1} \, : \, 2 \leq i \leq r \}.$$  The \emph{delta set} of $S$ is given as the union of delta sets of all non-identity elements:  $$\Delta(S) = \!\!\!\! \bigcup_{x \in S \setminus \{0\}} \!\!\!\! \Delta(x).$$
\end{defn}

\begin{example}\label{e:McNugget}
Consider the numerical monoid $M = \<6,9,20\>,$ which has Frobenius number $F(M) = 43$.  The element $60 \in M$ has factorization set $$\mathsf Z(60) = \{(0,0,3), (1,6,0), (4,4,0), (7,2,0),(10,0,0)\}.$$  This produces the length set $\mathsf L(60) = \{3,7,8,9,10\}$ and delta set $\Delta(60) = \{1,4\}$.  
\end{example}

While the delta set of a numerical monoid $S$ is the union of delta sets of all of its (infinitely many) non-identity elements, one needs only to compute $\Delta(x)$ for finitely many $x \in S$.  A weaker version of Theorem~\ref{t:deltaperiodic}, which appeared in~\cite{deltaperiodic}, states that delta sets in a numerical monoid are eventually periodic and that $\Delta(S)$ is obtained by taking the union of $\Delta(n)$ over every $n \le 2kn_2n_k^2 + n_1n_k$.  More recently, the authors of~\cite{compasympdelta} give a greatly improved bound $N_S$ (significantly smaller than the bound from~\cite{deltaperiodic}) for the start of this periodic behavior, and show that $\Delta(S)$ can be computed by taking the union of $\Delta(n)$ over every $n \le N_S + n_k - 1$.  

\begin{thm}\cite[Corollary~18]{compasympdelta}\label{t:deltaperiodic}
Fix a numerical monoid $S = \<n_1, \ldots, n_k\>$.  There exists an integer $N_S$ such that $n \ge N_S$ implies $\Delta(n) = \Delta(n + \lcm(n_1,n_k))$.  
\end{thm}

\begin{example}\label{e:deltaperiodic}
Figure~\ref{f:deltaperiodic} plots the delta sets of elements in $M = \<6,9,20\>$.  The bound given in \cite{compasympdelta} for the start of periodic behavior is $N_S = 144$, and we can see from the plot that the actual start of this behavior is 91.  Additionally, the eventual period is 20, which by~\cite{compasympdelta} is guaranteed to divide (but clearly need not equal to) $\lcm(6,20) = 60$.  Using this, we can compute the delta set of $M$ to be $\Delta(M) = \{1,2,3,4\}$.
\end{example}

\begin{figure}
\includegraphics[width=5.6in]{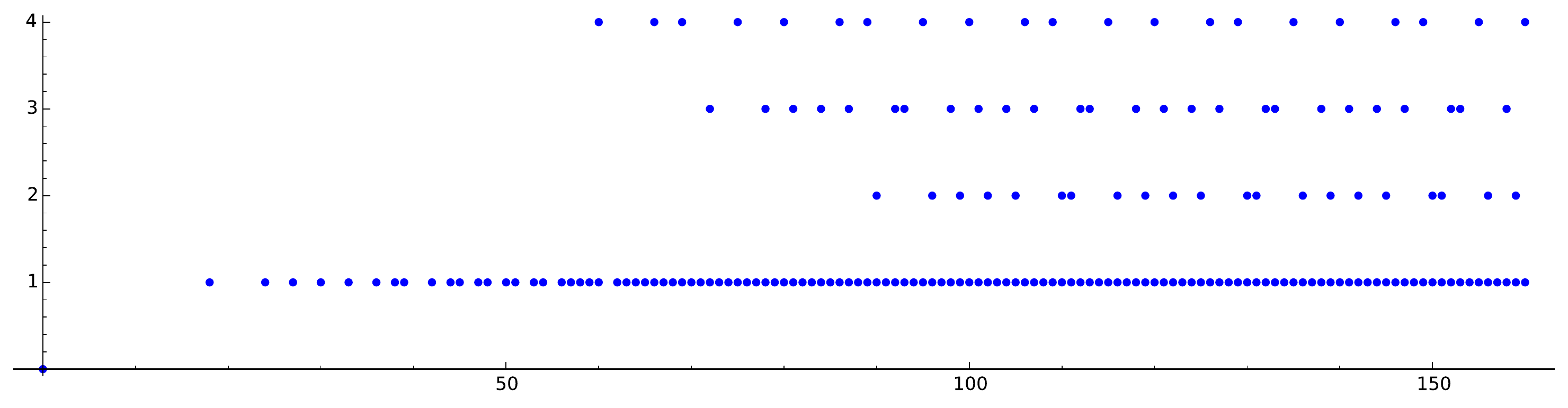}
\medskip
\caption{A plot showing the delta sets of elements in the numerical monoid $S = \<6,9,20\>$ from Example~\ref{e:omegaext}.  Here, a dot is placed at the point $(n,d)$ whenever $d \in \Delta(n)$.}  
\label{f:deltaperiodic}
\end{figure}

\section{Factorization sets and delta sets in numerical monoids}\label{s:dynamicfactor}

This section provides dynamic algorithms to compute factorization sets, length sets, and delta sets for numerical monoids (Algorithms~\ref{a:dynamicfactor},~\ref{a:dynamiclengths}, and ~\ref{a:dynamicdelta}, respectively), along with corresponding proofs of correctness.  Furthermore, we demonstrate the significant runtime advantages of these dynamic algorithms when compared to existing (non-dynamic) algorithms.  

We begin with a lemma describing an inductive relationship between the factorization sets of elements.  As this result applies to a significantly larger class of monoids than numerical monoids, it is stated in full generality.  In what follows, $\ee_i$ denotes the $i$th unit vector in $\NN^k$.

\begin{lemma}\label{l:dynamicfactor}
Fix a reduced, finitely generated monoid $M$ (written additively) with irreducible elements $m_1, m_2, \ldots, m_k$.  For each non-zero $x \in M$, we have 
$$\begin{array}{rcl}
\mathsf Z(x) &=& \bigcup_{i = 1}^k~\{\aa + \ee_i : \aa \in \mathsf Z(x - m_i)\} \\
&=& \bigsqcup_{i = 1}^k~\{\aa + \ee_i : \aa \in \mathsf Z(x - m_i), a_j = 0 \text{ for each } j < i\}.
\end{array}$$

\end{lemma}

\begin{proof}
Fix $\aa \in \mathsf Z(x)$, and let $i = \min\{j : a_j \ne 0\}$.  Then $\aa - \ee_i \in \mathsf Z(x-m_i)$.  
Moreover, if $\aa - \ee_j \in \mathsf Z(x-m_j)$ for some $j$, then $j \ge i$ by minimality of $i$.  
This proves the second equality, from which the first equality follows.  
\end{proof}

\begin{example}\label{e:dynamicfactor}
Let $M = \<6,9,20\>$.  Lemma~\ref{l:dynamicfactor} allows us to compute $\mathsf Z(60)$ in terms of $\mathsf Z(40)$, $\mathsf Z(51)$, and $\mathsf Z(54)$.  In particular, each factorization of $54$ yields a factorization of $60$ with one additional copy of the irreducible $6$.  Similarly, each factorization of $51$ or $40$ yields a factorization for $60$ with one additional copy of $9$ or $20$, respectively.  We give the full computation below.  
$$\begin{array}{rcl}
\mathsf Z(40) &=& \{(0,0,2)\}, \\
\mathsf Z(51) &=& \{(1,5,0),(4,3,0),(7,1,0)\}, \\
\mathsf Z(54) &=& \{(0,6,0),(3,4,0),(6,2,0),(9,0,0)\}, 
\vspace{0.2cm} \\
\mathsf Z(60) &=& \{(0,0,3)\} \,\cup\, \{(1,6,0),(4,4,0),(7,2,0)\} \\
&& \,{}\cup{}\, \{(1,6,0),(4,4,0),(7,2,0),(10,0,0)\} \\
&=& \{(0,0,3),(1,6,0),(4,4,0),(7,2,0),(10,0,0)\}.
\end{array}$$
Since each factorization of $60$ has at least one irreducible, the above computation produces the entire set $\mathsf Z(60)$.  Given  $n \ge 0$, Algorithm~\ref{a:dynamicfactor} uses this method inductively to compute $\mathsf Z(i)$ for all $i \in \{0, 1,  \ldots, n\}$, with $\mathsf Z(0) = \{\mathbf 0\}$ as its starting point.  
\end{example}

\begin{alg}\label{a:dynamicfactor}
Given $n \in S = \<n_1, \ldots, n_k\>$, computes $\mathsf Z(m)$  for all $m \in [0, n] \cap S$.  
\begin{algorithmic}
\Function{FactorizationsUpToElement}{$S$, $n$}
\State $F[0] \gets \{\mathbf 0\}$
\ForAll{$m \in [0, n] \cap S$}
	\State $Z \gets \{\}$
	\ForAll{$i = 1, 2, \ldots, k$ with $m - n_i \in S$}
		\State $Z \gets Z \cup \{\aa + \ee_i : \aa \in F[m - n_i]\}$
	\EndFor
	\State $F[m] \gets Z$
\EndFor
\State \Return $F$
\EndFunction
\end{algorithmic}
\end{alg}

\begin{remark}
Table~\ref{tb:dynamicfactor} gives a runtime comparison for Algorithm~\ref{a:dynamicfactor} with the existing implementation of $\mathsf Z(n)$ in the GAP package \texttt{numericalsgps}.  Notice that Algorithm~\ref{a:dynamicfactor} is slower than the current GAP implementition for computing a single factorization set $\mathsf Z(n)$, but faster for computing a large collection of factorization sets, such as $\{(m,\mathsf Z(m)) : m \le n\}$.  The same holds true for memory usage: Algorithm~\ref{a:dynamicfactor} consumes more memory than the GAP implementation when computing a single factorization set, since it must store several factorization sets, but when computing $\{(m,\mathsf Z(m)) : m \le n\}$, it does not consume any more memory than is necessary for the returned list.  
\end{remark}

\begin{table}
{\footnotesize
\begin{tabular}{|l|r|r|c|l|l|}
$S$ & $n$ & $|\mathsf Z(n)|$ & $\mathsf Z(n)$ & $\{(m,\mathsf Z(m)) : m \le n\}$ & Alg.~\ref{a:dynamicfactor} \\
\hline
$\<10,17,19,25,31\>$ & $1000$ & $20293$ & 641 ms & 55342 ms & 20120 ms \\
$\<51,53,55,117\>$ & $5000$ & $1299$ & 54 ms & 61643 ms & 3874 ms \\
$\<7,15,17,18,20\>$ & $1000$ & $75375$ & 1028 ms & 234339 ms & 102857 ms \\
$\<100,121,142,163,284\>$ & $30000$ & $16569$ & 2437 ms & 13266033 ms & 660320 ms \\
\end{tabular}
}
\medskip
\caption{Runtimes for computing $\mathsf Z(n)$ and $\{(m,\mathsf Z(m)) : m \le n\}$ given $n \in S$.  The last column gives runtimes for Algorithm~\ref{a:dynamicfactor} implemented in GAP, and the two previous columns use the implementation of $\mathsf Z(n)$ found in the GAP package \texttt{numericalsgps} \cite{numericalsgps}.}
\label{tb:dynamicfactor}
\end{table}

Typically, a length set for an element is computed only once its factorization set is known.  Lemma~\ref{l:dynamiclengths} provides a dynamic algorithm that allows the computation of the length set of an element simply by knowing the lengths sets of smaller elements (and not their factorization sets).  

\begin{lemma}\label{l:dynamiclengths}
Fix a reduced, finitely generated monoid $M$ with irreducible elements $m_1, m_2, \ldots, m_k$.  We have 
$$\mathsf L(x)
= \bigcup_{i = 1}^k \left(\mathsf L(x-m_i) + 1\right)
= \bigcup_{i = 1}^k~\{l + 1 : l \in \mathsf L(x-m_i)\}$$
for each nonzero $x \in M$.  
\end{lemma}

\begin{proof}
Apply the length function to both sides of the first equality of Lemma~\ref{l:dynamicfactor}.  
\end{proof}

\begin{example}\label{e:dynamiclengths}
We resume notation from Example~\ref{e:dynamicfactor}.  Lemma~\ref{l:dynamiclengths} specializes the computation of $\mathsf Z(n)$ in Lemma~\ref{l:dynamicfactor} to length sets, allowing us to compute $\mathsf L(60)$ from $\mathsf L(40)$, $\mathsf L(51)$, and $\mathsf L(54)$.  The key observation is that producing a factorization for $60$ from a factorization of $40$, $51$ or $54$ using Lemma~\ref{l:dynamicfactor} always increases length by exactly one.  Given below is the full computation.  
$$\begin{array}{rcl}
\mathsf L(60) &=& \left(\mathsf L(40) + 1\right) \,\cup\, \left(\mathsf L(51) + 1\right) \,\cup\, \left(\mathsf L(54) + 1\right) \\
&=& \{3\} \,\cup\, \{7,8,9\} \,\cup\, \{7,8,9,10\} \\
&=& \{3,7,8,9,10\}
\end{array}$$
Given $n \ge 0$, Algorithm~\ref{a:dynamiclengths} uses this method to dynamically compute $\mathsf L(m)$ for all $m \in [0, n] \cap S$, which Algorithm~\ref{a:dynamicdelta} then uses to compute $\bigcup_{m = 1}^n \Delta(m)$.  Notice that this does not require computing any factorizations; see Remark~\ref{r:deltacomparison}.  
\end{example}

\begin{alg}\label{a:dynamiclengths}
Given $n \in S = \<n_1, \ldots, n_k\>$, computes $\mathsf L(m)$ for all $m \in [0, n] \cap S$.  
\begin{algorithmic}
\Function{LengthSetsUpToElement}{$S$, $n$}
\State $\mathcal L[0] \gets \{\mathbf 0\}$
\ForAll{$m \in [0, n] \cap S$}
	\State $\mathsf L \gets \{\}$
	\ForAll{$i = 1, 2, \ldots, k$ with $m - n_i \in S$}
		\State $\mathsf L \gets \mathsf L \cup \{l + 1 : l \in \mathcal L[m - n_i]\}$
	\EndFor
	\State $\mathcal L[m] \gets L$
\EndFor
\State \Return $\mathcal L$
\EndFunction
\end{algorithmic}
\end{alg}

\begin{remark}\label{r:dynamicdelta}
Since Algorithm~\ref{a:dynamiclengths} computes $\{\mathsf L(m) : m \in [0, n] \cap S\}$, it can also be used to compute $\bigcup_{m = 1}^n \Delta(m)$ by first computing $\{\Delta(m) : m \in [0, n] \cap S\}$.  As stated in Theorem~\ref{t:deltaperiodic}, $\Delta(m)$ is periodic for $m$ greater than an integer $N_S$ described in~\cite{compasympdelta}.  Together with Algorithm~\ref{a:dynamiclengths}, this yields Algorithm~\ref{a:dynamicdelta} for computing $\Delta(S)$.  
\end{remark}

\begin{alg}\label{a:dynamicdelta}
Given a numerical monoid $S$, computes $\Delta(S)$.  
\begin{algorithmic}
\Function{DeltaSet}{$S$}
\State Compute $\mathcal L$ using Algorithm~\ref{a:dynamiclengths}
\State Compute $N_S$ using \cite[Section~3]{compasympdelta}
\State Compute $\Delta = \bigcup_{m \in (0, N_S + \lcm(n_1,n_k)] \cap S} \Delta(m)$
\State \Return $\Delta$
\EndFunction
\end{algorithmic}
\end{alg}

\begin{remark}\label{r:deltacomparison}
In addition to giving an improved bound $N_S$ on the start of periodic behavior of $\Delta_S$ for any numerical monoid $S = \<n_1, \ldots, n_k\>$, the authors also give \cite[Algorithm~21]{compasympdelta} to find $\Delta(S)$ by first computing $\mathsf Z(N_S + n_k - n_1), \ldots, \mathsf Z(N_S + n_k - 1)$.  Table~\ref{tb:dynamicdelta} gives a runtime comparison between Algorithm~\ref{a:dynamicdelta} and \cite[Algorithm~21]{compasympdelta}, demonstrating a marked improvement in computation time.  

Our method of computing $\Delta(S)$ has several advantages over \cite[Algorithm~21]{compasympdelta}.  First and foremost is memory comsumption.  For large $n$, factorization sets grow large very quickly \cite{factorasymp}, whereas $|\mathsf L(n)|$ grows linearly in $n$ \cite[Theorem~4.3]{elastsets}.  Since \cite[Algorithm~21]{compasympdelta} requires computing $\mathsf Z(n)$ for several large $n$, it is often memory intensive.  Algorithm~\ref{a:dynamicdelta}, on the other hand, avoids the computation of factorization sets altogether by only computing length sets.  

Due in part to the low memory footprint, our algorithm is significantly faster than \cite[Algorithm~21]{compasympdelta}.  Some runtimes for \cite[Algorithm~21]{compasympdelta} are omitted from Table~\ref{tb:dynamicdelta}, as the high memory requirements left us unable to complete the computation.  However, \cite[Table~2]{compasympdelta} gives a runtime for the computation of $\Delta(\<31,73,77,87,91\>)$ on the order of 24,000 seconds, a stark contrast to the 4.2 seconds required for Algorithm~\ref{a:dynamicdelta}.  
\end{remark}

\begin{remark}\label{r:deltaringbuffer}
Although Algorithm~\ref{a:dynamicdelta} as stated requires the computation of $\mathsf L(i)$ for all $i \le N_S$, Algorithm~\ref{a:dynamiclengths} computes each $\mathsf L(i)$ using values at least $i - n_k$.  As such, in implementing Algorithm~\ref{a:dynamicdelta} to compute $\Delta(S)$, one only needs to store $n_k$ length sets at any given time.  The implemenation of Algorithm~\ref{a:dynamicdelta} in GAP stores the length sets in a ring buffer of length $n_k$, cutting the memory requirements even further.  
\end{remark}

\begin{remark}\label{r:newdeltaalgs}
Since the implementation of Algorithm~\ref{a:dynamicdelta} in the \texttt{numericalsgps} package~\cite{numericalsgps}, other promising delta set algorthms have been developed.  Recent results of the second author of this manuscript \cite{factorhilbert} have produced an algorithm that computes the delta set of any affine monoid (a strictly more general setting) and appears to run faster than Algorithm~\ref{a:dynamicdelta}.  This algorithm is currently being implemented in GAP and will likely be included in a future version of the \texttt{numericalsgps} package.  Additionally, a log-time algorithm to compute $\Delta(S)$ is given in \cite{deltadim3} for the special case where $S = \<n_1, n_2, n_3\>$ and $S$ is non-symmetric.  
\end{remark}

\begin{table}
{\footnotesize
\begin{tabular}{|l|r|r|c|l|l|}
$S$ & $N_S$ \cite{compasympdelta} & Dis. & $\Delta(S)$ & \cite[Alg.~21]{compasympdelta} & Alg.~\ref{a:dynamicdelta} \\
\hline
$\<10,17,19,25,31\>$ & $1180$ & $76$ & $\{1,2,3\}$ & 3254 ms & 45 ms \\
$\<51,53,55,117\>$ & $9699$ & $1006$ & $\{2,4,6\}$ & 23565 ms & 250 ms \\
$\<7,15,17,18,20\>$ & $1935$ & $46$ & $\{1,2,3\}$ & 88831 ms & 146 ms \\
$\<7,19,20,25,29\>$ & $3894$ & $76$ & $\{1,2,3,5\}$ & -------- ms & 624 ms \\
$\<11,53,73,87\>$ & $14381$ & $873$ & $\{2,4,6,8,10,22\}$ & 49418 ms & 2588 ms \\
$\<31,73,77,87,91\>$ & $31364$ & $558$ & $\{2,4,6\}$ & -------- ms & 4274 ms \\
$\<100,121,142,163,284\>$ & $24850$ & $5499$ & $\{21\}$ & -------- ms & 3697 ms \\
$\<1001,1211,1421,1631,2841\>$ & $2063141$ & $114535$ & $\{10,20,30\}$ & -------- ms & 116371 ms \\
\end{tabular}
}
\medskip
\caption{Runtime comparison for computing $\Delta(S)$.  Some computations for \cite[Algorithm~21]{compasympdelta} could not be completed due to insufficient available memory and hence have been omitted.  All computations were completed using GAP and the package \texttt{numericalsgps} \cite{numericalsgps}.}
\label{tb:dynamicdelta}
\end{table}

\section{$\omega$-primality in the quotient group $\mathsf q(M)$}\label{s:omegaext}

In the remaining sections, we turn our attention from factorization sets and delta sets to $\omega$-primality (Definition~\ref{d:omegaext}), a factorization invariant that has received much attention in recent investigations~\cite{andalg,omegaquasi,omegamonthly,compasympomega}.  The results presented in these sections are motivated by those in Section~\ref{s:dynamicfactor}, namely that dynamic computations can yield significant performance improvements when running a large collection of computations.  

In this section, we provide the necessary definitions and motivation related to the $\omega$-primality invariant.  We begin with a definition of the quotient group of a monoid.


\begin{remark}\label{r:quotientgroup}
Let $\mathsf q(M)$ denote the quotient group of a monoid $M$.  The map $M \to \mathsf q(M)$ given by $m \mapsto (m - 0)$ is an injective monoid homomorphism, and we often identify elements of $M$ with their image in $\mathsf q(M)$ and view $M \subset \mathsf q(M)$.  Under this convention, $x \mid y$ for elements $x, y \in \mathsf q(M)$ if $y = a + x$ for some $a \in M$.  
\end{remark}

The extension of the $\omega$-function on a monoid to elements of its quotient group is one of the key insights of this paper.  In this section, we develop the theory behind this extension, with minimal assumptions on the underlying monoid.  

\begin{defn}\label{d:omegaext}
Fix a monoid $M$.  The \emph{$\omega$-primality function} $\omega_M: \mathsf q(M) \to \NN \cup \{\infty\}$ is given by $\omega_M(x) = m$ if $m$ is the smallest positive integer with the property that whenever $\sum_{i = 1}^r a_i - x \in M$ for $r > m$ and $a_i \in M$, there exists a subset $T \subset \{1, \ldots, r\}$ with $|T| \le m$ such that $\sum_{i \in T} a_i - x \in M$.  If no such $m$ exists, define $\omega_M(x) = \infty$.  When $M$ is clear from context, we simply write~$\omega(x)$.  
\end{defn}

\begin{remark}\label{r:omegaext}
It is easy to check that the function $\omega_M$ defined above coincides with that defined in \cite[Definitions~2.3 and~3.4]{omegamonthly} for elements of $M$, so Definition~\ref{d:omegaext} simply extends the domain of $\omega_M$ from $M$ to $\mathsf q(M)$.  Example~\ref{e:omegaext} demonstrates that this is a natural extension, and the results that follow show that many of the properties of the usual $\omega$-function still hold in this new setting.  
\end{remark}

\begin{remark}\label{r:omegamult}
Written in a multiplicative setting, Definition~\ref{d:omegaext} becomes more transparent; see \cite[Definition~2.3]{omegamonthly} for more detail.  In particular, is clear that prime monoids elements $x \in M$ are precisely those satisfying $\omega(x) = 1$.   In fact, much of the seminal work using $\omega$-primality focused on computing $\omega$-values for irreducible elements, as non-unique factorizations arise from irreducible elements that are not prime.
\end{remark}

\begin{example}\label{e:omegaext}
Let $S = \<6,9,20\>$ denote the numerical monoid from Example~\ref{e:dynamicfactor}.  Since $S$ has finite complement in $\NN$, $\mathsf q(S)$ is naturally isomorphic to $\ZZ$, with the obvious inclusion map.  Figure~\ref{f:omegaext} plots side-by-side the $\omega$-values of $S$ (as defined in \cite{omegamonthly}) and those of $\mathsf q(S)$ (from Definition~\ref{d:omegaext}).  Notice the plotted values coincide for each $n \in S$, and the $\omega_S$-values in the right-hand plot of elements lying in the complement of $S$ seem to ``fill in'' the missing values in the left-hand plot.  
\end{example}

\begin{figure}
\includegraphics[width=2.5in]{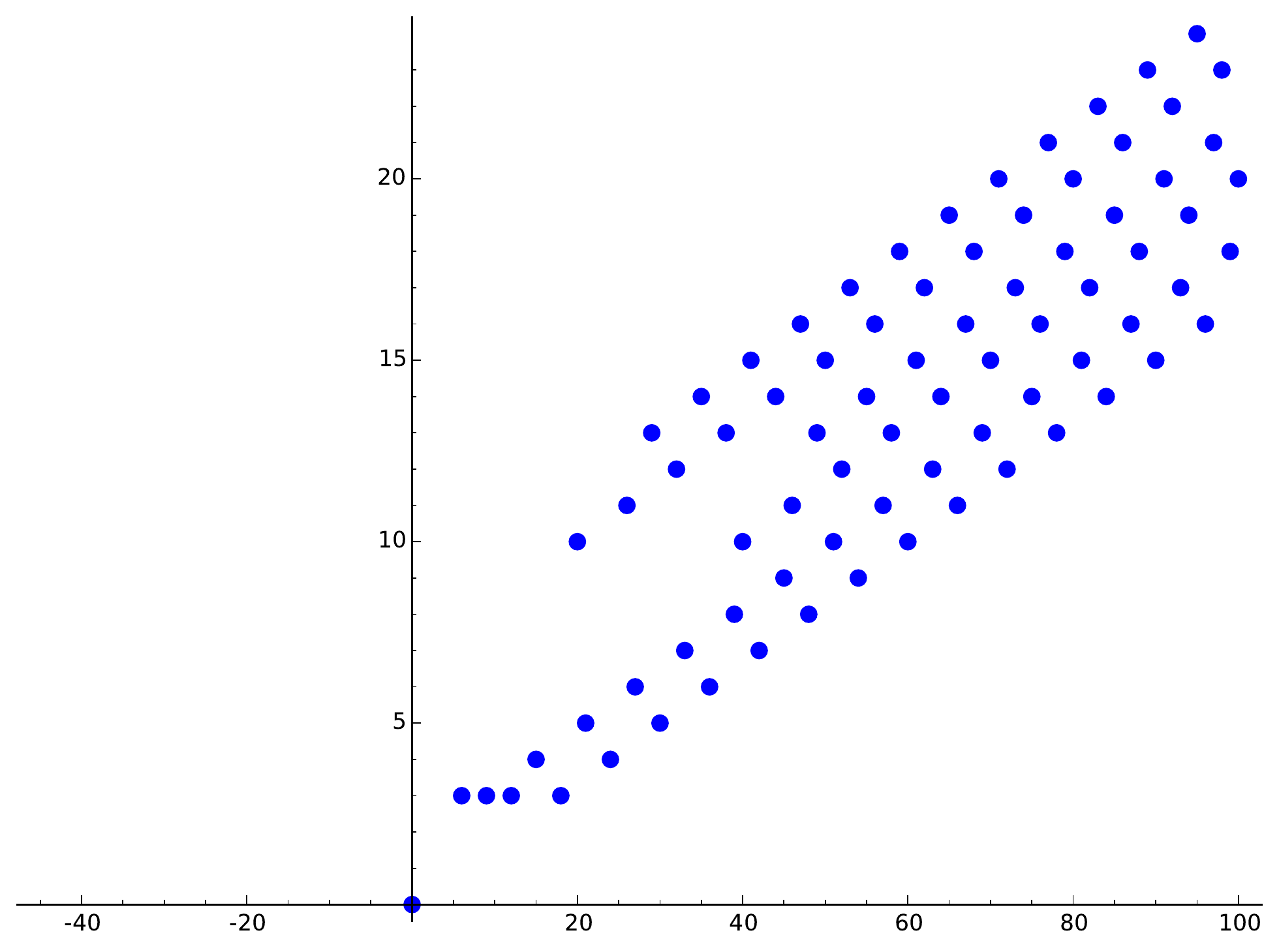}
\hspace{0.4in}
\includegraphics[width=2.5in]{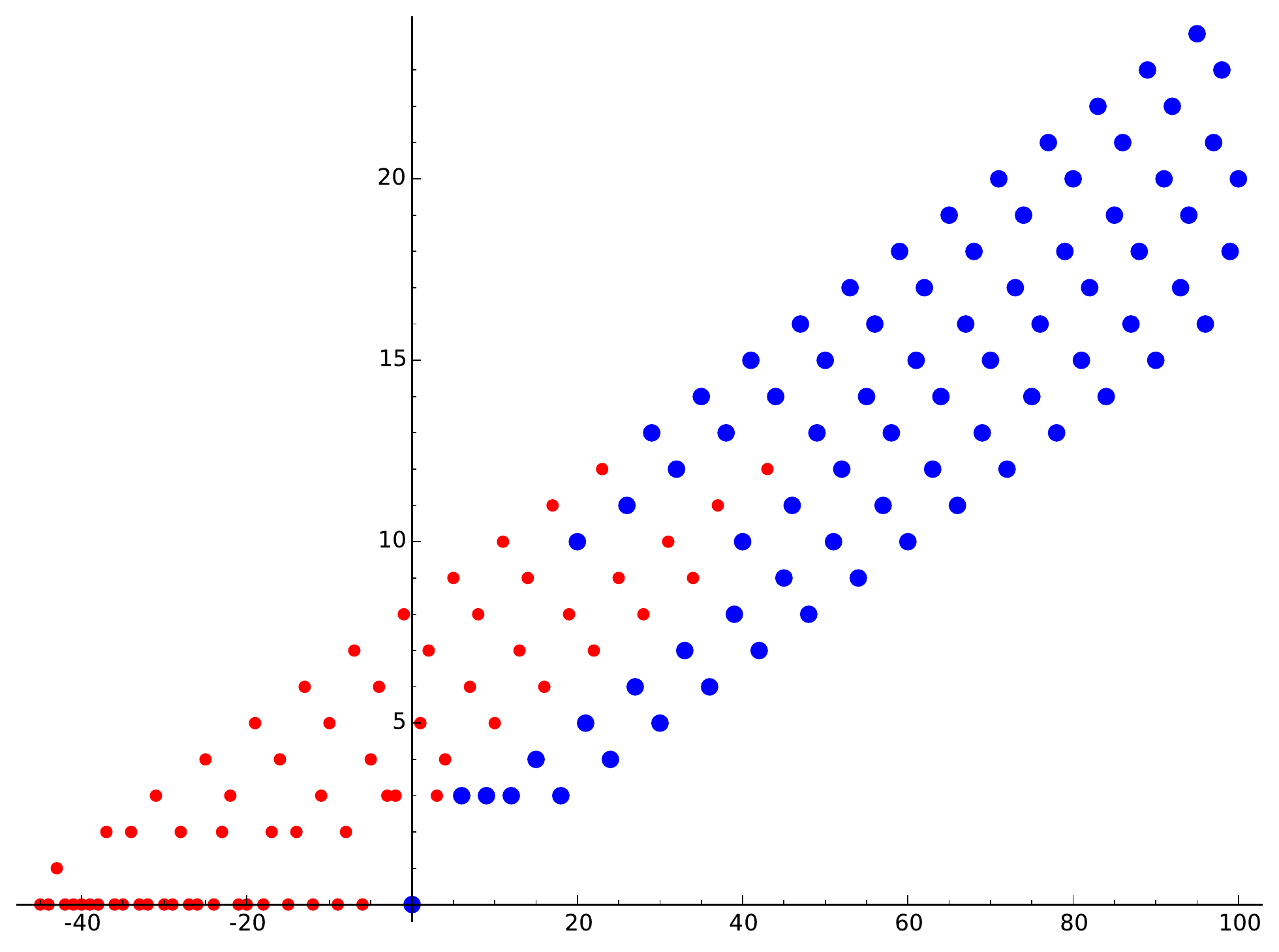}
\medskip
\caption{A plot of the $\omega$-values of elements in the numerical monoid $S = \<6,9,20\>$ (left) and its quotient group (right) discussed in Example~\ref{e:omegaext}.  The smaller red dots in the right-hand plot mark $\omega$-values of elements lying outside of $S$.}  
\label{f:omegaext}
\end{figure}

One of the key ideas used to study $\omega$-primality in recent years is its characterization in terms of bullet lengths \cite[Proposition~2.10]{omegamonthly}.  We now extend the definition of bullets to elements of the quotient group (Definition~\ref{d:bulletext}) and recover the usual characterization of $\omega$-primality in terms of their length (Proposition~\ref{p:bulletext}).  Note that Proposition~\ref{p:bulletext} also implies that we may assume the elements $a_i$ in Defintion~\ref{d:omegaext} are irreducible.  

\begin{defn}\label{d:bulletext}
Fix a monoid $M$.  A \emph{bullet} for $x \in \mathsf q(M)$ is an expression $u_1 + \cdots + u_r$ of irreducible elements $u_1, \ldots, u_r \in M$ such that (i) $u_1 + \cdots + u_r - x \in M$, and (ii) $u_1 + \cdots + u_r - x - u_i \notin M$ for each $i \le r$.  The \emph{value} of a bullet $u_1 + \cdots + u_r$ is the element $u_1 + \cdots + u_r \in M$, and its \emph{length} is $r$.  The set of bullets of $x$ is denoted $\bul(x)$.  
\end{defn}

\begin{remark}\label{r:bulletext}
For a monoid element $m \in M$, the set $\bul(m)$ in Definition~\ref{d:bulletext} is identical to the usual definition (\cite[Definition~2.8]{omegamonthly}), since any element of $\mathsf q(M)$ that $m$ divides also must lie in $M$.  Both of these definitions differ slightly from the classical definition of a bullet; see~\cite{semitheor}.  
\end{remark}

\begin{remark}\label{r:bullettuple}
Fix a monoid $M$ and an element $x \in M$.  If, additionally, $M$ is both reduced (that is, $M$ has no non-identity units) and finitely generated, then $M$ has only finitely many irreducible elements $u_1, \ldots, u_k$.  In this setting, we can denote a bullet $b_1u_1 + \cdots + b_ku_k \in \bul(x)$ by $\bb = (b_1, \ldots, b_k) \in \NN^k$.  This is of particular use in Sections~\ref{s:dynamicomega} and~\ref{s:omegabound} (as well as Example~\ref{e:covermaps}), where $\omega$-primality for numerical monoids is examined in more detail.  
\end{remark}

\begin{prop}\label{p:bulletext}
Given any monoid $M$, 
$$\omega(x) = \sup\{r : u_1 + \cdots + u_r \in \bul(x), u_i \emph{\text{ irreducible}}\}$$
for each element $x \in \mathsf q(M)$.  
\end{prop}

\begin{proof}
This is identical to the proof of~\cite[Proposition~2.10]{omegamonthly}.  
\end{proof}

The notion of cover maps between bullet sets, first introduced in \cite[Definition~3.4]{omegaquasi}, will be critical in the dynamic algorithm for computing $\omega$-primality (Algorithm~\ref{a:dynamicomega}).  In the remainder of this section, we extend the original definition of cover maps to elements of the quotient group (Definition~\ref{d:covermaps}) and provide stronger results than the prior setting would allow (see Remark~\ref{r:covermaps}).  First, we give Lemma~\ref{l:subsume}, which was stated in~\cite{andalg} for numerical monoids and played a crucial role in developing the first algorithm to compute $\omega$-primality.  

\begin{lemma}\label{l:subsume}
Fix a monoid $M$, an element $x \in \mathsf q(M)$, a bullet $u_1 + \cdots + u_r \in \bul(x)$, and an irreducible element $u \in M$.  
Then $u + u_1 + \cdots + u_r \notin \bul(x)$.  
\end{lemma}

\begin{proof}
Omitting $u$ yields the expression $u_1 + \cdots + u_r$, which $x$ divides.   
\end{proof}

Much like Lemma~\ref{l:dynamicfactor} for factorizations, the bullet set of an element $x$ is determined by the bullet sets of its divisors.   Definition~\ref{d:covermaps} specifies how to determine the image of each bullet as demonstrated in Example~\ref{e:covermaps}, and Theorem~\ref{t:covermaps} ensures the resulting map is well-defined.  First, we give an example.  

\begin{example}\label{e:covermaps}
Let $S = \<6,9,20\>$ denote the numerical monoid from Example~\ref{e:omegaext}.  Consider $n = 60$ and the following bullet sets: 
$$\begin{array}{rcl}
\bul(40) &=& \{(0,0,2),(4,4,0),(7,2,0),(10,0,0),(1,6,0),(0,8,0)\}, \\
\bul(51) &=& \{(0,7,0),(10,0,0),(4,3,0),(1,5,0),(0,0,3),(7,1,0)\}, \\
\bul(54) &=& \{(9,0,0),(6,2,0),(0,6,0),(3,4,0),(0,0,3)\}, \\
\bul(60) &=& \{(4,4,0),(7,2,0),(10,0,0),(1,6,0),(0,8,0),(0,0,3)\}
\end{array}$$
We see that for each bullet $\bb \in \bul(54)$, either $\bb \in \bul(60)$ or $\bb + \ee_1 \in \bul(60)$.  Notice that it is impossible for both of these to lie in $\bul(60)$ by Lemma~\ref{l:subsume}.  Similarly, for each bullet $\bb \in \bul(51)$, either $\bb \in \bul(60)$ or $\bb + \ee_2 \in \bul(60)$, and for each $\bb \in \bul(40)$, either $\bb \in \bul(60)$ or $\bb + \ee_3 \in \bul(60)$.  Moreover, each bullet for $60$ is the ``image'' of a bullet for $54$, $51$ or $40$ in this way.  The resemblance to Lemma~\ref{l:dynamicfactor} here is not a coincidence; Theorem~\ref{t:omegafactor} makes this similarity precise.  
\end{example}

\begin{thm}\label{t:covermaps}
Let $M$ be a monoid.  Fix an element $x \in \mathsf q(M)$, an irreducible $u \in M$, and a bullet $u_1 + \cdots + u_r \in \bul(x)$.  
\begin{enumerate}
\item[(i)] If $u_1 + \cdots + u_r - (u + x)\in M$, then $u_1 + \cdots + u_r \in \bul(u + x)$.
\item[(ii)] If $u_1 + \cdots + u_r - (u + x) \notin M$, then $u + u_1 + \cdots + u_r \in \bul(u + x)$.  
\end{enumerate}
\end{thm}

\begin{proof}
Suppose $u_1 + \cdots + u_r - (u + x) \in M$.  For each $i \le r$, we also have 
$$u_1 + \cdots + u_r - (u_i + u + x) = (u_1 + \cdots + u_r - x - u_i) - u \notin M$$
since $u_1 + \cdots + u_r - (u_i + x) \notin M$.  This means $u_1 + \cdots + u_r \in \bul(u + x)$.  

Next, suppose $u_1 + \cdots + u_r - (u + x) \notin M$.  We now verify the necessary conditions.  
\begin{enumerate}
\item[(i)] Since $u_1 + \cdots + u_r \in \bul(x)$, we have $u + u_1 + \cdots + u_r - (u + x) = u_1 + \cdots + u_r - x \in M$.  
\item[(ii)] For each $i \le r$, $u + u_1 + \cdots + u_r - (u_i + u + x) = u_1 + \cdots + u_r - (u_i - x) \notin M$.  
\end{enumerate}
Thus, $u + u_1 + \cdots + u_r \in \bul(u + x)$.  
\end{proof}

\begin{defn}\label{d:covermaps}
Fix a monoid $M$, and an irreducible $u \in M$.  For $x \in \mathsf q(M)$, the \emph{$u$-cover map} $\psi_M^u:\bul(x) \to \bul(u + x)$ is given by 
$$\psi_M^u(u_1 + \cdots + u_r) = \left\{\begin{array}{ll}
u_1 + \cdots + u_r & \text{if } u_1 + \cdots + u_r - (u + x) \in M \\
u + u_1 + \cdots + u_r & \text{otherwise} 
\end{array}\right.$$
for each $u_1 + \cdots + u_r \in \bul(x)$.  When there is no confusion, we often omit the subscript.  
\end{defn}

\begin{remark}\label{r:covermaps}
The cover map $\phi_M^u$ was also defined in~\cite{omegaquasi} in the context of numerical monoids, though its domain was restricted to bullets in which $u$ appears.  Both the extended domain of $\psi_M^u$ in Definition~\ref{d:covermaps} and the extended domain of $\omega_M$ in Definition~\ref{d:omegaext} are crucial for the results that follow.  In fact, Theorem~\ref{t:bulletscovered} and Corollary~\ref{c:bulletscovered} only held for ``sufficiently large'' elements of $M$ prior to making these extentions.  These stronger statements are our primary motivation for extending $\omega_M$ to $\mathsf q(M)$, as they ensure the correctness of Algorithm~\ref{a:dynamicomega} for all (numerical) monoid elements.  
\end{remark}

\begin{thm}\label{t:bulletscovered}
Fix a monoid $M$, $x \in \mathsf q(M)$, and a bullet $u_1 + \cdots + u_r \in \bul(x)$.  For each $j \le r$, $u_1 + \cdots + u_r - u_j \in \bul(x - u_j)$.  
\end{thm}

\begin{proof}
Since $u_1 + \cdots + u_r \in \bul(x)$, we have 
\begin{itemize}
\item[(i)] $(u_1 + \cdots + u_r - u_j) - (x - u_j) = u_1 + \cdots + u_r - x \in M$, and 
\item[(ii)] $(u_1 + \cdots + u_r - u_j) - (u_i + x - u_i) = u_1 + \cdots + u_r - (u_i + x) \notin M$ for each $i \ne j$.  
\end{itemize}
This means $u_1 + \cdots + u_r - u_j \in \bul(x - u_j)$.  
\end{proof}

Corollary~\ref{c:bulletscovered} follows directly from Theorems~\ref{t:covermaps} and~\ref{t:bulletscovered}, and will serve as the inductive step of Algorithm~\ref{a:dynamicomega}.  It also justifies use of the term ``cover map''.  

\begin{cor}\label{c:bulletscovered}
If $M$ is a monoid, $x \in \mathsf q(M)$, and $u_1 + \cdots + u_r \in \bul(x)$, then
$$\bul(x) = \bigcup_{i \le r} \psi_M^{u_i}(\bul(x - u_i)).$$
\end{cor}

We conclude the section with Proposition~\ref{p:omegazero}, which drastically simplifies the base case for Algorithm~\ref{a:dynamicomega} (see Remark~\ref{r:basecase} for more detail).  

\begin{prop}\label{p:omegazero}
Fix a monoid $M$.  For each $x \in \mathsf q(M)$, the following are equivalent: 
\begin{enumerate}
\item[(i)] $\omega(x) = 0$.  
\item[(ii)] $\bul(x) = \{\mathbf 0\}$.  
\item[(iii)] $\mathbf 0 \in \bul(x)$.  
\item[(iv)] $-x \in M$.  
\end{enumerate}
\end{prop}

\begin{proof}
We show (i)${}\implies{}$(ii)${}\implies{}$(iii)${}\implies{}$(i) and (iii)${}\Leftrightarrow{}$(iv).  

(i)${}\implies{}$(ii): 
If $\omega(x) = 0$, then any bullet for $x$ has length at most 0, so $\bul(x) = \{\mathbf 0\}$.  

(ii)${}\implies{}$(iii): 
This is clear.  

(iii)${}\implies{}$(i): 
This follows from Lemma~\ref{l:subsume}.  

(iii)${}\Leftrightarrow{}$(iv): 
This follows directly from Definition~\ref{d:omegaext}.  
\end{proof}

\begin{example}\label{e:omegazero}
Let $S = \<6,9,20\>$ denote the numerical monoid in Example~\ref{e:omegaext}.  Proposition~\ref{p:omegazero} implies $\omega_S(n) = 0$ for every $n < -43 = -F(S)$; see Figure~\ref{f:omegaext}.  
\end{example}

\section{A dynamic programming algorithm for $\omega$-primality}\label{s:dynamicomega}

In this section, we apply the results of Section~\ref{s:omegaext} in the setting of numerical monoids to obtain Algorithm~\ref{a:dynamicomega} for dynamically computing $\omega$-values, with significant runtime improvements over existing algorithms.  Since any numerical monoid $S$ has finite complement in $\NN$, we have $\mathsf q(S) \iso \ZZ$ and write $\ZZ$ in place of $\mathsf q(S)$ in what follows.  

\begin{remark}\label{r:basecase}
One of the main difficulties in developing Algorithm~\ref{a:dynamicomega} was determining where to start the inductive process.  While Algorithms~\ref{a:dynamicfactor} and~\ref{a:dynamiclengths} naturally begin at the irreducible elements of $S$, the $\omega$-value and bullet set of an irreducible element $x$ are only trivial when $x$ is prime, and this is never the case for non-trivial numerical monoids (see Proposition~\ref{p:omegapseudofrob}).  One possible (though inelegant) solution is to use existing algorithms to compute the bullet sets of every element below some threshold, and then use a dynamic algorithm for all further values.  

Thankfully, the extended notion of $\omega$-primality presented in Section~\ref{s:omegaext} allows us to avoid this caveat entirely.  Proposition~\ref{p:dynamicomega}a states that for any numerical monoid $S$, the smallest element of $\mathsf q(S) = \ZZ$ on which $\omega_S$ takes a nonzero value is $-F(S)$, providing a natural base case for Algorithm~\ref{a:dynamicomega}.  
\end{remark}

As a non-trivial numerical monoid $S$ contains no prime elements, $\omega_S$ only takes on values strictly greater than $1$ on elements in $S$.  Proposition~\ref{p:omegapseudofrob}, however, characterizes which elements of $\mathsf q(S)$ have an $\omega_S$-value of 1.  See \cite{redeisfiniteness,numerical,pseudofrob} for more background on pseudo-Frobenius numbers.  


\begin{prop}\label{p:omegapseudofrob}
Fix a numerical monoid $S = \<n_1, \ldots, n_k\>$.  For each $n \in \ZZ$, following are equivalent.  
\begin{enumerate}
\item[(i)] $\omega(n) = 1$.  
\item[(ii)] $\bul(n) = \{\ee_1, \ldots, \ee_k\}$.  
\item[(iii)] $-n$ is a pseudo-Frobenius number of $S$.  
\end{enumerate}
\end{prop}

\begin{proof}
We show (i)${}\Leftrightarrow{}$(ii) and (ii)${}\Leftrightarrow{}$(iii).  

(i)${}\implies{}$(ii): Suppose $\omega(n) = 1$.  For each $i \in \{1, \ldots, k\}$, we have $b_i \ee_i \in \bul(n)$ for some $b_i \ge 0$ since numerical monoids are Archimedian, and by Proposition~\ref{p:omegazero}, $b_i > 0$.  Since $\omega(n) = 1$, each $b_i = 1$.  

(ii)${}\implies{}$(i): This follows from Proposition~\ref{p:bulletext}.  

(ii)${}\implies{}$(iii): Suppose $\bul(n) = \{\ee_1, \ldots, \ee_k\}$.  For each $i \in \{1, \ldots, k\}$, we have $n_i - n \in S$ and $n_i - n - n_i = -n \notin S$, meaning $-n$ is a pseudo-Frobenius number of $S$.  

(iii)${}\implies{}$(ii): If $-n$ is a psuedo-Frobenius number of $S$, then $-n \notin S$ and $n_i - n \in S$ for each $i \in \{1, \ldots, k\}$, so $\ee_i \in \bul(n)$.  Lemma~\ref{l:subsume} ensures no other bullets exist.  
\end{proof}

\begin{remark}
Bullets in numerical monoids are usually denoted either as tuples or as a list of irreducibles, the former of which is well-suited for storage when implementing algorithms.  However, for the purposes of implementing Algorithm~\ref{a:dynamicomega}, each bullet $\aa = (a_1, \ldots, a_k) \in \bul(n)$ can be represented using only its value $v = a_1n_1 + \cdots + a_kn_k$ and length $\ell = a_1 + \cdots + a_k$ (a \emph{dynamic bullet}, see Definition~\ref{d:bulletdyn}).  Indeed, $v$ is sufficient for determining the image of $\aa$ under cover maps, and $\ell$ is sufficient for computing $\omega(n)$ once all bullets have been found.  This is the content of Proposition~\ref{p:dynamicomega}.  

Since several distinct bullets can be represented by the same dynamic bullet, the cardinality of the computed bullet set with this representation is significantly reduced.  Much like dynamically computing length sets instead of factorization sets greatly improves efficiency when computating delta sets (Remark~\ref{r:deltacomparison}), using this compact bullet representation significantly reduces the runtime and memory footprint when dynamically computing $\omega$-values.  
\end{remark}

Proposition~\ref{p:dynamicomega} is the analog of Lemmas~\ref{l:dynamicfactor} and~\ref{l:dynamiclengths} for Algorithm~\ref{a:dynamicomega}.  First and foremost, parts~(b) and~(c) prove the correctness of Algorithm~\ref{a:dynamicomega} by ensuring that dynamic bullets (Definition~\ref{d:bulletdyn}) are sufficient for computing $\omega$-values.  Moreover, part~(a) provides a natural starting place for the inductive procedure in Algorithm~\ref{a:dynamicomega}.  Lastly, part~(d) proves that we may further restrict our attention to maximal dynamic bullets, and Corollary~\ref{c:constanttime} demonstrates the benefit of this reduction (see Remark~\ref{r:constanttime}).  

\begin{defn}\label{d:bulletdyn}
Fix a numerical monoid $S = \<n_1, \ldots, n_k\>$ and an element $n \in \mathsf q(S)$.  A \emph{dynamic bullet for $n$} is an ordered pair $(v,\ell) \in \NN^2$ such that $v = \sum_{i = 1}^k b_in_i$ and $\ell = \sum_{i = 1}^k b_i$ are the value and length of some $\bb \in \bul(n)$, respectively.  The set of dynamic bullets of $n$ is denoted $\bul^*(n)$.  The \emph{value} and \emph{length} of a dynamic bullet $(v,\ell) \in \bul^*(n)$ are the values $v$ and $\ell$, respectively, and $(v,\ell)$ is \emph{maximal} if $\ell$ is maximal among dynamic bullets with value $v$.  
\end{defn}

\begin{prop}\label{p:dynamicomega}
Fix a numerical monoid $S = \<n_1, \ldots, n_k\>$, and fix $i \in \{1, \ldots, k\}$ and $n \in~\mathsf q(S)$.  Let $\psi:\bul(n) \to \bul(n + n_i)$ denote the $n_i$-cover map, $\phi_n:\bul(n) \to \bul^*(n)$ denote the map given by
$$\bb \in \bul(n) \longmapsto \left(\textstyle\sum_{i = 1}^k b_in_i, \sum_{i = 1}^k b_i\right) \in \bul^*(n)$$
and $\psi^*:\bul^*(n) \to \bul^*(n + n_i)$ denote the map given by
$$(v,\ell) \in \bul^*(n) \mapsto \left\{\begin{array}{ll}
(v, \ell) & \text{if } v - (n_i + n) \in S \\
(v + n_i, \ell + 1) & \text{if } v - (n_i + n) \notin S.
\end{array}\right.$$
\begin{enumerate}
\item[(a)] If $n < -F(S)$, then $\bul(n) = \{\mathbf 0\}$.  
\item[(b)] For all $n \in \mathsf q(S)$, $\omega(n) = \max\{w : (v,w) \in \bul^*(n)\}$.  
\item[(c)] For all $n \in \mathsf q(S)$ and $i \in \{1, \ldots, k\}$, the following diagram commutes: 
$$\arraycolsep=1.4pt\begin{array}{ccc}
\bul(n) &\xrightarrow{\qquad\psi\qquad}& \bul(n + n_i) \\
\left\downarrow \rule{0pt}{0.4in} {}^{\phi_n} \right. && \left\downarrow \rule{0pt}{0.4in} {}^{\phi_{n+n_i}} \right. \\
\bul^*(n) &\xrightarrow{\qquad\psi^*\qquad}& \bul^*(n + n_i)
\end{array}$$
\item[(d)] A dynamic bullet $(v,\ell)$ is maximal if and only if $\psi^*(v,\ell)$ is maximal.  
\end{enumerate}
\end{prop}

\begin{proof}
First, Theorem~\ref{t:covermaps} ensures that the map $\psi^*$ is well-defined upon observing that the value and length of the image of a bullet $\bb$ under $\psi$ depend only on the value and length of $\bb$.  Now, parts~(a) and~(b) follow directly from Propositions~\ref{p:omegazero} and~\ref{p:bulletext}, respectively, and part~(c) follows immediately from Definitions~\ref{d:covermaps} and~\ref{d:bulletdyn}.  Lastly, part~(d) follows from part~(c) and the observation that applying $\psi^*$ to dynamic bullets $(v,\ell_1), (v, \ell_2) \in \bul^*(n)$ preserves any ordering on the second component.  
\end{proof}

\begin{alg}\label{a:dynamicomega}
Given $n \in S = \<n_1, \ldots, n_k\>$, computes $\omega(m)$ for all $m \in [0, n] \cap S$.  
\begin{algorithmic}
\Function{OmegaPrimalityUpToElement}{$S$, $n$}
\State $B_m \gets \{(0,0)\}$ for all $m < -F(S)$
\ForAll{$m \in \{-F(S), \ldots, n\}$}
	\State $B_m \gets \{\}$
	\ForAll{$i \in \{1, 2, \ldots, k\}$}
		\ForAll{$(v, \ell) \in B_{m - n_i}$}
			\State $B_m \gets B_m \cup 
			\left\{\begin{array}{ll}
			\{(v, \ell)\} & \text{if } v - (n_i + m) \in S \\
			\{(v + n_i, \ell + 1)\} & \text{otherwise}
			\end{array}\right.$
		\EndFor
	\EndFor
	\State $B_m \gets \{(v,\ell) \in B_m : (v,\ell) \text{ is maximal}\}$
	\If{$m \in S$}
		\State $\omega[m] \gets \max\{\ell : (v, \ell) \in B_m\}$
	\EndIf
\EndFor
\State \Return $\omega$
\EndFunction
\end{algorithmic}
\end{alg}

\begin{remark}
Table~\ref{tb:omegafactor} gives a runtime comparison between Algorithm~\ref{a:dynamicomega} and \cite[Proposition~3.3]{semitheor}, the algorithm previously implemented in the GAP package \texttt{numericalsgps}.  Not only is Algorithm~\ref{a:dynamicomega} significantly faster, it also computes an extensive list of $\omega$-values, rather than just a single value.  
\end{remark}

\section{Computing $\omega$-primality from factorizations}\label{s:omegabound}

There are many similarities between Algorithms~\ref{a:dynamicfactor} and~\ref{a:dynamicomega} beyond their dynamic nature.  In non-precise terms, bullet sets and factorization sets seem to behave in a very similar manner.  In this section, we make this connection explicit for numerical monoids $S$ by characterizing the bullet set $\bul(n)$ of any element $n \in S$ in terms of factorization sets of certain submonoids of $S$.  As an application, we greatly improve the existing bound \cite{omegaquasi} on the start of quasilinear behavior (Theorem~\ref{t:omegaquasi}) of the $\omega$-function on $S$; see Theorem~\ref{t:newomegabound}.  

\subsection{Runtime of Algorithm~\ref{a:dynamicomega}}

We first provide the definition of an Ap\'ery set, which will be utilized in Theorem~\ref{t:omegafactor} below.  See~\cite{numerical} for a full treatment of Ap\'ery sets.  

\begin{defn}\label{d:apery}
Let $M$ be a monoid.  The \emph{Ap\'ery set} of $x \in M$ is defined as 
$$\Ap(M,x) = \{m \in M : m-x \in \qq(M) \setminus M\}.$$
\end{defn}

\begin{thm}\label{t:omegafactor}
Fix a reduced finitely generated monoid $M$, and let $G = \{u_1, \ldots, u_k\}$ denote the set of irreducible elements of $M$.  For $A \subset G$ nonempty and $x \in M$, write 
$$\mathsf Z_A(x) = \{\aa \in \mathsf Z(x) : a_i = 0 \text{ whenever } u_i \notin A\} \subset \mathsf Z(x),$$
that is, the factorizations in $\mathsf Z_M(x)$ corresponding to factorizations in $\mathsf Z_{\<A\>}(x)$.  Then 
$$\bul(x) = \left\{\bb \in \mathsf Z_{A}(y + x) : \emptyset \ne A \subset G \text{ and } y \in \textstyle\bigcap_{u_i \in A} \Ap(M;u_i) \right\}$$
for all $x \in \mathsf q(M)$ with $\omega(x) > 0$.  
\end{thm}

\begin{proof}
First, fix $A \subset G$ nonempty, $y \in \bigcap_{u_i \in A} \Ap(M;u_i)$, and $\bb \in \mathsf Z_A(y+x)$.  It follows that $\sum_{j=1}^k b_ju_j - x = y \in M$, and 
$$\sum_{j=1}^k b_ju_j - (u_i + x) = y - u_i \notin M$$
for each $u_i \in A$, so $\bb \in \bul(x)$.  

Now, fix $\bb \in \bul(x)$.  Let $A = \{u_i : b_i \ne 0\} \subset G$ and $y = \sum_{j=1}^k b_ju_j - x$.  Since $\omega(x) > 0$, $A$ is nonempty by Proposition~\ref{p:omegazero}.  Rearranging the equation for $y$ gives $y + x = \sum_{j=1}^k b_ju_j$, meaning $\bb \in \mathsf Z_A(y + x)$, and since $y \in M$ and $y - u_i \notin M$ for $b_i \ne 0$, we have $y \in \bigcap_{u_i \in A} \Ap(M;u_i)$.  
\end{proof}

\begin{example}\label{e:omegafactor}
Let $S = \<6,9,20\>$ and $n = 60 \in S$.  Our goal is to compute $\bul(60)$ using Theorem~\ref{t:omegafactor}.  For any subset $A \subset \{6,9,20\}$ such that $\gcd(A) > 1$ and $60 \in \<A\>$, each $m \in \bigcap_{a \in A} \Ap(S;a)$ is relatively prime to $\gcd(A)$, so $\mathsf Z_A(60 + m) = \emptyset$.  The only subsets not satisfying both of these conditions are $A = \{9\}$ and $A = \{6,9,20\}$.  Since $\Ap(S;6) \cap \Ap(S;9) \cap \Ap(S;20) = \{0\}$ and the only nonzero $m \in \Ap(S;9)$ with $\mathsf Z_{\{9\}}(60 + m)$ nonempty is $m = 12$, we compute $\bul(60)$ as follows: 
$$\begin{array}{rcl}
\bul(60)
&=& \mathsf Z_{\{6,9,20\}}(60 + 0) \cup \mathsf Z_{\{9\}}(60 + 12) \\
&=& \{(4,4,0),(7,2,0),(10,0,0),(1,6,0),(0,0,3)\} \cup \{(0,8,0)\}. \\
\end{array}$$
\end{example}

\begin{table}
\begin{tabular}{|l|l|l|l|l|}
$S$ & $n$ & $\omega(n)$ & \cite[Prop.~3.3]{semitheor} & Algorithm~\ref{a:dynamicomega} \\
\hline
$\<6,9,20\>$ & $1000$ & 170 & 61319 ms & 6 ms \\
$\<11,13,15\>$ & $1000$ & 97 & 10732 ms & 5 ms \\
$\<11,13,15\>$ & $3000$ & 279 & 874799 ms & 15 ms \\
$\<11,13,15\>$ & $10000$ & 915 & -------- & 42 ms \\
$\<15,27,32,35\>$ & $1000$ & 69 & 234738 ms & 9 ms \\
$\<10,12,15,16,17\>$ & $500$ & 52 & 14338290 ms & 40 ms \\
$\<10,12,15,16,17\>$ & $50000$ & 5002 & -------- & 4084 ms \\
$\<100,121,142,163,284\>$ & $25715$ & 308 & -------- & 27449 ms \\
$\<1001,1211,1421,1631,2841\>$ & $357362$ & 405 & -------- & 3441508 ms \\
\end{tabular}
\medskip
\caption{Runtime comparison for computing the $\omega$-value of $n \in \mathsf q(S)$.  All computations performed using the GAP package \texttt{numericalsgps} \cite{numericalsgps} and do not make use of eventual quasilinearity.}
\label{tb:omegafactor}
\end{table}

\begin{remark}\label{r:factoromegacompare}
For each element $n$ in a numerical monoid $S$, Theorem~\ref{t:omegafactor} gives a way to compute the set $\bul(n)$ (and thus $\omega(n)$) by computing factorizations of elements in certain submonoids of $S$.  Since factorization sets in numerical monoids can be computed rather quickly in GAP \cite{numericalsgps}, this yields an algorithm for computing $\omega$-primality.  This turns out to be faster than \cite[Proposition~3.3]{semitheor}, the algorithm widely used to compute $\omega$-primality in numerical monoids prior to Algorithm~\ref{a:dynamicomega}, but not as fast as Algorithm~\ref{a:dynamicomega}, as it still relies on computing $\bul(n)$, rather than the (substantially smaller) set $\bul^*(n)$.  Additionally, algorithmic use of Theorem~\ref{t:omegafactor} only produces a single $\omega$-value, whereas Algorithm~\ref{a:dynamicomega} produces an exhaustive list of $\omega$-values.   
\end{remark}

\begin{remark}\label{r:constanttime}
One immediate consequence of Theorem~\ref{t:omegafactor} is that the number of distinct values occuring for dynamic bullets in $\bul^*(n)$ is at most $|\bigcup_{i = 1}^k \Ap(S;n_i)|$.  In particular, this bound does not depend on $n$.  Since Algorithm~\ref{a:dynamicomega} only stores a single dynamic bullet with each value (Proposition~\ref{p:dynamicomega}d), the inductive step runs in constant time in~$n$ (for fixed $S$).  This yields Corollary~\ref{c:constanttime}.  
\end{remark}

\begin{cor}\label{c:constanttime}
For fixed $S$, Algorithm~\ref{a:dynamicomega} runs in linear time in $n$.  
\end{cor}

\subsection{Bounding the eventual quasilinearity of $\omega$-primality}

For numerical monoids, the $\omega$-function admits a predictable behavior for sufficiently large values, a result that appeared as \cite[Theorem~3.6]{omegaquasi} and independently as~\cite[Corollary~20]{compasympomega}.  Before stating this result here as Theorem~\ref{t:omegaquasi}, we give a definition.  

\begin{defn}\label{d:quasi}
A map $f: \NN \to \NN$ is \emph{quasilinear} if $f(n) = a_1(n) n + a_0(n)$, 
where $a_1, a_2: \NN \to \QQ$ are each periodic and $a_1(n)$ is not identically zero.  
\end{defn}

\begin{thm}\cite[Theorem~3.6]{omegaquasi}\label{t:omegaquasi}
Fix a numerical monoid $S = \<n_1, \ldots, n_k\>$.  The $\omega$-primality function is eventually quasilinear.  In particular, there is a periodic function $a(n)$ with period $n_1$  and integer $N_0$ such that for $n > N_0$, 
$$\omega_S(n) = \frac{1}{n_1}n + a(n).$$
\end{thm}

\begin{remark}\label{r:quasi}
The value of $N_0$ described in Theorem~\ref{t:omegaquasi}, which gives an upper bound for the start of the quasilinear behavior of $w$, is explicitly given in ~\cite{omegaquasi}.  The actual start of the quasilinear behavior, which is called the \emph{dissonance}, is very often significantly lower.  Theorem~\ref{t:newomegabound} will greatly improve the bound for the dissonance.
\end{remark}

\begin{example}\label{e:omegaquasi}
Figure~\ref{f:omegaquasi} plots the $\omega$-values for the numerical monoids $\<4,13,19\>$ and $\<6,9,20\>$.  The eventual quasilinearity ensured by Theorem~\ref{t:omegaquasi} manifests graphically as a collection of discrete lines with identical slopes, and is readily visible from these plots.  The proof of Theorem~\ref{t:omegaquasi} appearing in \cite{omegaquasi} gives a precise (but computationally impractical) bound on the start of this quasilinear behavior, and Theorem~\ref{t:newomegabound} below drastically improves this bound; see Table~\ref{tb:newomegabound} for a detailed comparison.  
\end{example}

\begin{figure}
\includegraphics[width=2.5in]{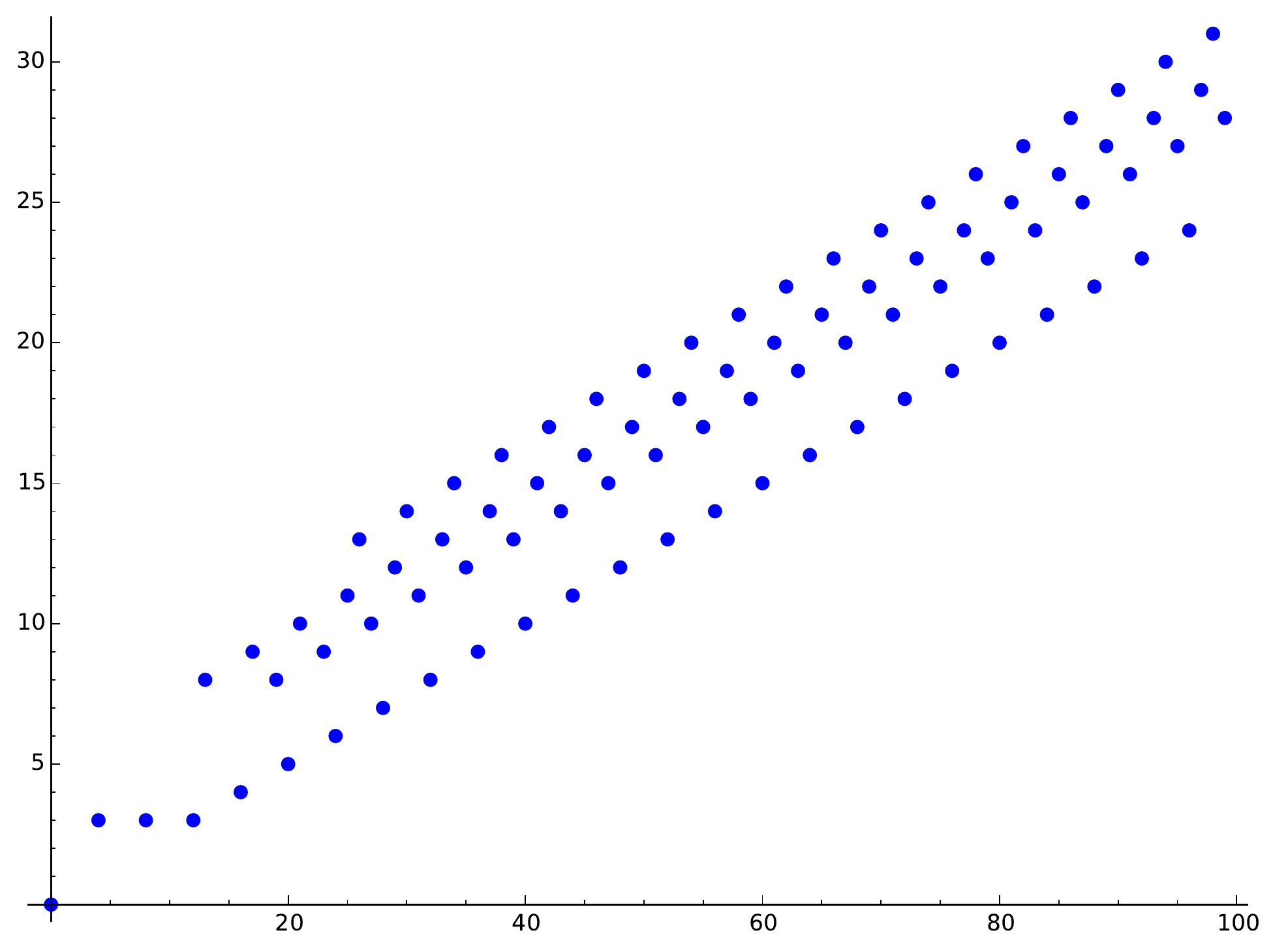}
\hspace{0.5in}
\includegraphics[width=2.5in]{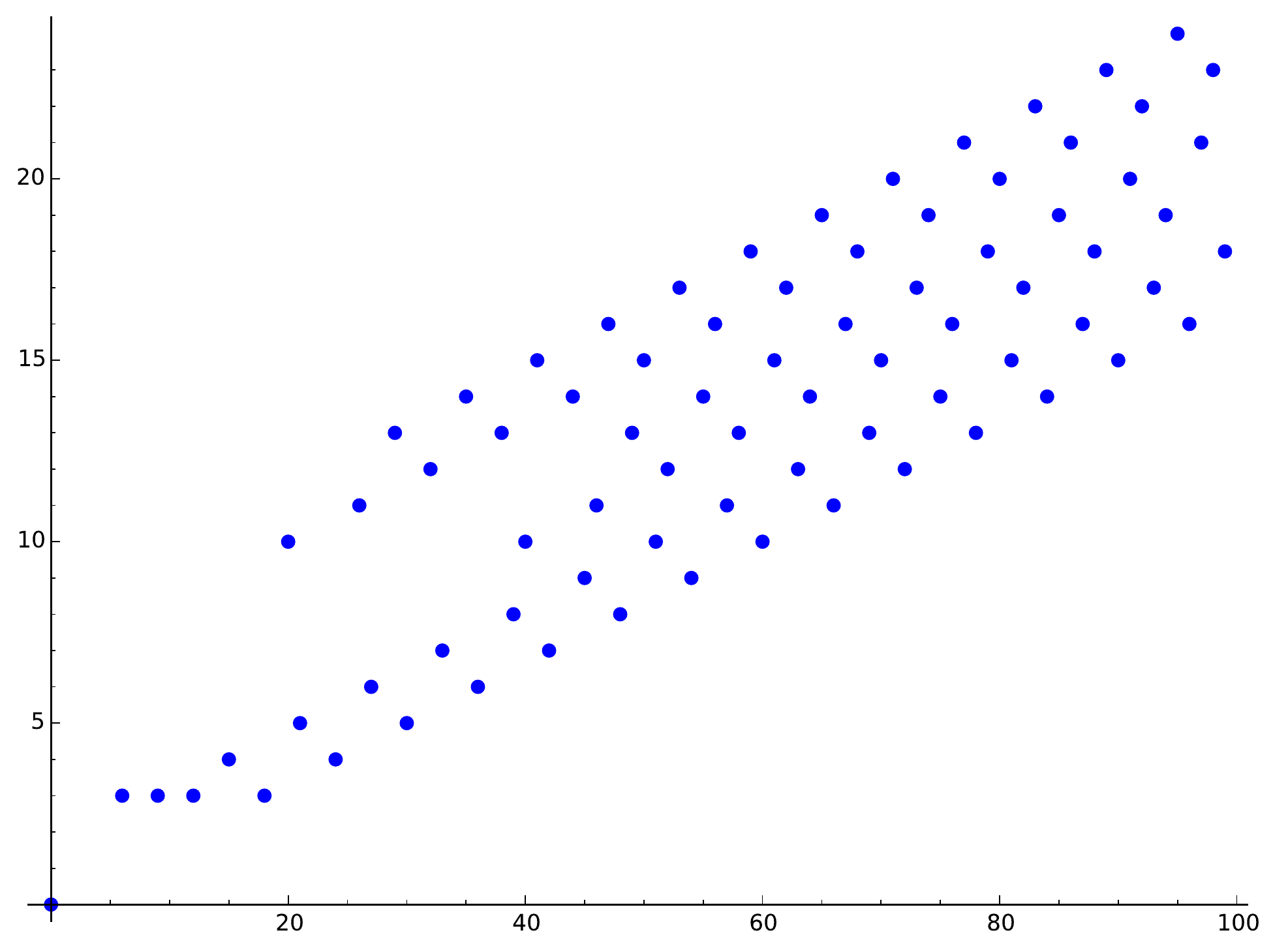}
\medskip
\caption{A plot showing the function $\omega_S$ for the numerical monoid $S = \<4,13,19\>$ (left) and $\<6,9,20\>$ (right), discussed in Example~\ref{e:omegaquasi}.}  
\label{f:omegaquasi}
\end{figure}

We now apply Theorem~\ref{t:omegafactor} to improve the bound on the dissonance point of the $\omega$-function.  As with the original bound on the dissonance point given in \cite{omegaquasi}, this result is heavily motivated by computational evidence, in this case provided by Algorithm~\ref{a:dynamicomega}.  

\begin{lemma}\label{l:newomegabound}
Fix a numerical monoid $S = \langle n_1, n_2, \ldots, n_k \rangle$.  For each $n \in S$, we have 
$$M(n) \le n/n_1 \le \omega(n),$$ where $M(n)$ is the maximum length of a factorization for $n$.
\end{lemma}

\begin{proof}
Fix $b_1 > 0$ such that $\bb = b_1 \ee_1 \in \bul(n)$ (note that $b_1$ exists since numerical monoids are Archimedian).  This gives $\omega(n) \ge |\bb| \ge n/n_1$.  
Additionally, we have $n = \sum_{j=1}^k a_jn_j \ge a_1n_1$ for each factorization $\aa \in \mathsf Z(n)$, meaning $M(n) \le n/n_1$.  
\end{proof}

\begin{thm}\label{t:newomegabound}
Fix a numerical monoid $S = \<n_1, \ldots, n_k\>$, and suppose $n > N_0$, where 
$$N_0 = \frac{F(S) + n_2}{n_2/n_1 - 1}.$$
Any maximal bullet $\bb \in \bul(n)$ satisfies $b_1 > 0$, and $\omega(n) = \omega(n - n_1) + 1$.  
\end{thm}

\begin{proof}
Fix a bullet $\bb \in \bul(n)$, and let $A = \{n_i : b_i > 0\}$.  First, suppose $b_1 = 0$.  By Theorem~\ref{t:omegafactor}, we have $\bb \in \mathsf Z(n + s)$ for some $s \in \bigcap_{n_i \in A} \Ap(S;n_i)$.  In particular, $s - \min(A) \notin S$, so $s - \min(A) \le F(S)$.  
Since $n > N_0$, we have 
$$|\bb| \le M_{\<A\>}(n + s) \le \frac{n+s}{\min(A)} \le \frac{n + F(S)}{\min(A)} + 1 \le \frac{n + F(S)}{n_2} + 1 < n/n_1 \le \omega_S(n).$$
by Lemma~\ref{l:newomegabound}, meaning $\bb$ is not maximal.  This proves the first statement.  

Now, suppose $\bb$ is maximal, and let $\aa = \bb - \ee_1$.  The above argument implies $b_1 > 0$, so by Theorem~\ref{t:bulletscovered}, we have $\aa \in \bul(n - n_1)$.  This gives
$$\omega(n) - 1 \ge \omega(n - n_1) \ge |\aa| = |\bb| - 1 = \omega(n) - 1,$$
where the first inequality follows from Theorem~\ref{t:covermaps}(ii).  
\end{proof}

\begin{remark}\label{r:newomegabound}
The bound given in \cite{omegaquasi} for the dissonance point of the $\omega$-function of a numerical monoid $S$, while explicit, is far too large to be feasable reached in computation.  In contrast, the new bound $N_0$ given in Theorem~\ref{t:newomegabound} can actually be reached computationally; that is, it is often possible to compute every $\omega$-value below $N_0$ (using Algorithm~\ref{a:dynamicomega}, for instance).   See Table~\ref{tb:newomegabound} for a detailed comparison of these bounds.  The $\omega$-value of any element above $N_0$ is determined by the $\omega$-values of the elements just below $N_0$, and can be obtained by simply evaluating a linear function.  This makes it possible to compute the $\omega$-value of any element of $S$. 
\end{remark}

\begin{table}
\begin{tabular}{|l|l|l|l|}
$S$ & Dissonance & Theorem~\ref{t:newomegabound} & $n_0$ \cite[Theorem~3.6]{omegaquasi} \\
\hline
$\<6,9,20\>$ & 12 & 104 & 37,800 \\
$\<10,12,15\>$ & 190 & 325 & 66,600 \\
$\<10,12,15,16,17\>$ & 10 & 175 & 34,272,000 \\
$\<10,12,13,14,15,16,17,18,19,21\>$ & 10 & 115 & 450,087,321,600 \\
$\<100,121,142,163,284\>$ & 100 & 25715 & 64,426,520,664,000 \\
$\<1001,1211,1421,1631,2841\>$ & 1001 & 357362 & $\approx 6 \cdot 10^{19}$ \\
\end{tabular}
\medskip
\caption{Comparison for the start of quasilinear behavior of the $\omega$-function.  The dissonance points were computed using Algorithm~\ref{a:dynamicomega}, available in the next release of the GAP package \texttt{numericalsgps} \cite{numericalsgps}.}
\label{tb:newomegabound}
\end{table}

\section{Future Work}\label{s:futurework}

It is natural to ask whether the dynamic algorithms given above can be extended to settings outside the realm of numerical monoids.  In fact, Algorithms~\ref{a:dynamicfactor} and~\ref{a:dynamiclengths} extend immediately to any finitely generated monoid $M$, once a suitably ``bounded'' subset of $M$ is chosen (for instance, one could use the set of divisors of some element $m \in M$).  In order to generalize Algorithm~\ref{a:dynamicdelta}, the eventual behavior of $\Delta:M \to 2^\ZZ$ must be examined.  

\begin{prob}\label{pr:deltafingen}
Describe the eventual behavior of $\Delta:M \to 2^\ZZ$ for $M$ finitely generated.  
\end{prob}

It is worth noting that the generalization of Algorithm~\ref{a:dynamiclengths} to finitely generated monoids does allow the use of computation packages for investigating Problem~\ref{pr:deltafingen}.  

Generalizing Algorithm~\ref{a:dynamicomega} presents a slightly more subtle issue.  Section~\ref{s:omegaext} extends the $\omega$-function on a finitely generated monoid $M$ to its quotient group $\mathsf q(M)$, along with both the ``base case'' and the ``inductive step'' for Algorithm~\ref{a:dynamicomega}.  However, given an element $m \in M$, the set of divisors of $m$ with nonzero $\omega$-value need not be finite; see Example~\ref{e:omegacone}.  Thus, more care must be taken in order to dynamically compute $\omega$-values in $M$, and it is as yet unclear how to make this distinction.  

\begin{example}\label{e:omegacone}
Consider the monoid $M \subset \NN^2$ consisting of all points not lying above the ray generated by $(2,1)$.  Definition~\ref{d:omegaext} extends the domain of $\omega_M$ to all of $\ZZ^2$, and Proposition~\ref{p:omegazero} states that $\omega_M(m) = 0$ precisely when $m \in -M$.  However, the set of divisors of any nonzero element $m \in M$ contains infinitely many elements of $\ZZ^2$ which necessarily have both nonzero $\omega$-value and nontrivial bullet sets; see Figure~\ref{f:omegacone}.  
\end{example}

\begin{figure}
\includegraphics[width=3in]{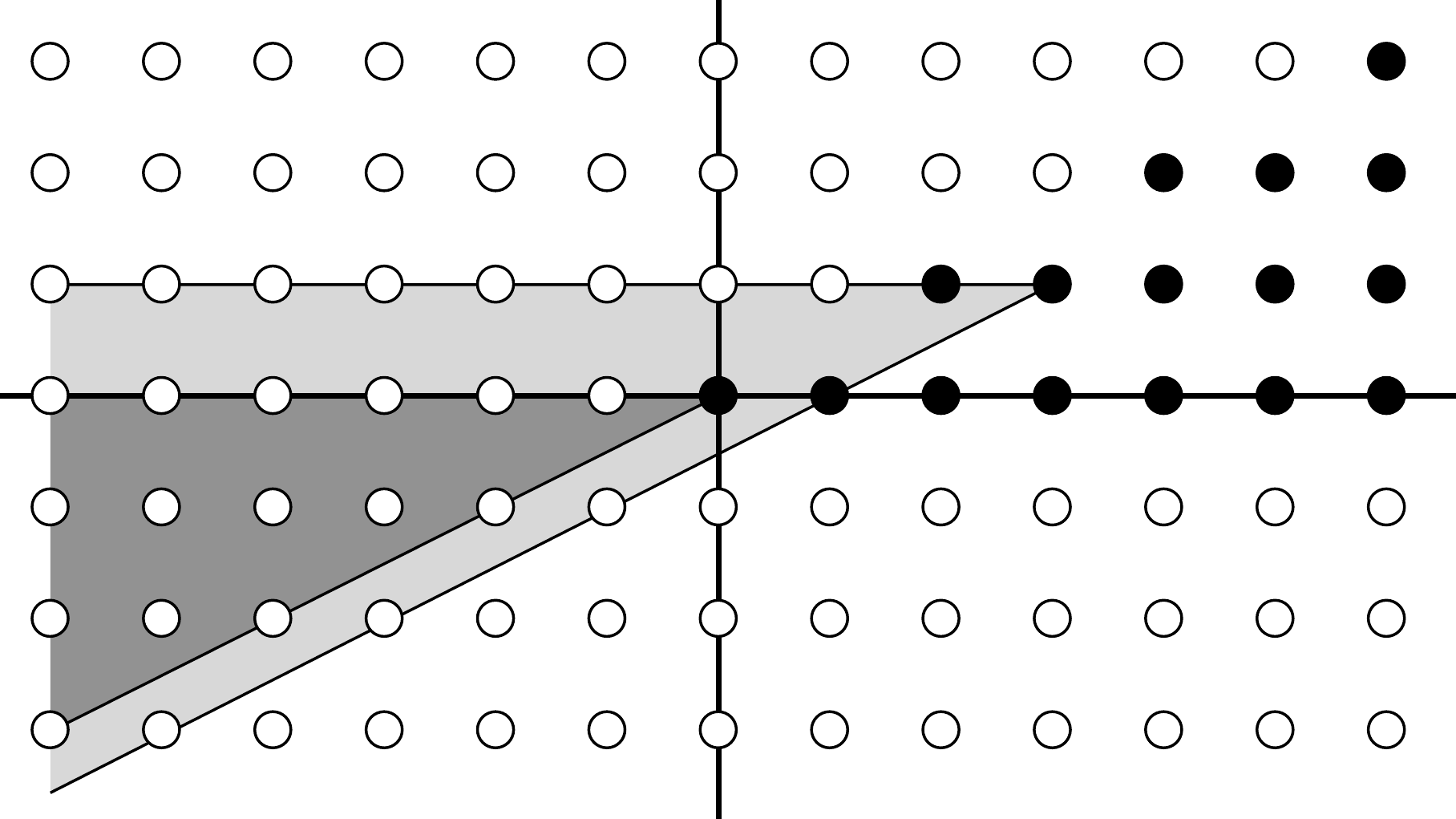}
\medskip
\caption{The monoid $M \subset \NN^2$ given in Example~\ref{e:omegacone}.  The elements of $M$ are marked by filled dots, the lightly shaded region marks the divisors of $m = (3,1)$, and the darkly region marks $-M$.}
\label{f:omegacone}
\end{figure}

\begin{prob}\label{pr:omegafingen}
Develop a dynamic programming algorithm to compute $\omega$-values for finitely generated monoids.  
\end{prob}

Given an element $n \in S$, its \emph{catenary degree} $c(n)$ can be computed from a graph with vertex set $\mathsf Z(n)$ \cite{catenarytame}.  Like delta sets, the catenary degree function $c:S \to \NN$ is periodic for large input values \cite{catenaryperiodic}, but no bound on the start of this periodic behavior is known, and the period is only known to divide the product of the generators of $S$.  The iterative nature of the periodicity proof given in \cite{catenaryperiodic} hints at the existence of a dynamic algorithm for computing catenary degrees in numerical monoids, allowing for the use of computation when investigating the eventual periodic behavior.  

\begin{prob}\label{pr:catenaryalg}
Find a dynamic programming algorithm to compute the catenary degrees of numerical monoid elements.  
\end{prob}



\begin{thebibliography}{HHHKR10}
\raggedbottom




\bibitem{andalg}
D. F. Anderson, S. Chapman, N. Kaplan, and D. Torkornoo, 
\emph{An algorithm to compute $\omega$-primality in a numerical monoid}, 
Semigroup Forum 82 (2011), no. 1, 96--108.

\bibitem{elastsets} 
 T. Barron, C. O'Neill, and R. Pelayo, 
 \emph{On the set of elasticities in numerical monoids}, 
 preprint, available at \textsf{arXiv:1409.3425}.  To appear in \emph{Semigroup Forum}.  

\bibitem{semitheor}
V.Blanco, P. Garc\'ia-S\'anchez, and A. Geroldinger,
\emph{Semigroup-theoretical characterizations of arithmetical invariants with applications to numerical monoids and Krull monoids},
Illinois Journal of Mathematics, 55 \textbf{4} (2011), 1385--1414

\bibitem{delta}
C. Bowles, S. Chapman, N. Kaplan, and D. Reiser, 
\emph{On delta sets of numerical monoids}, 
J. Algebra Appl.  5(2006), 1--24.

\bibitem{catenaryperiodic}
S.~Chapman, M.~Corrales, A.~Miller, C.~Miller, and D.~Patel,
\emph{The catenary and tame degrees on a numerical monoid are eventually periodic},
J. Aust. Math. Soc. 97 (2014), no. 3, 289--300.

\bibitem{catenarytame}
S. Chapman, P. Garc\'ia-S\'anchez, and D. Llena, 
\emph{The catenary and tame degree of numerical semigroups}, 
Forum Math 21(2009), 117--129.

\bibitem{deltaperiodic}
S. Chapman, R. Hoyer, and N. Kaplan, 
\emph{Delta sets of numerical monoids are eventually periodic}, 
Aequationes mathematicae 77 \textbf{3} (2009) 273--279.



\bibitem{numericalsgps}
M. Delgado, P. Garc\'ia-S\'anchez, and J. Morais, 
\emph{GAP Numerical Semigroups Package},
\url{http://www.gap-system.org/Manuals/pkg/numericalsgps/doc/manual.pdf}.

\bibitem{redeisfiniteness}
P.~Freyd,
\emph{R\'edei's finiteness theorem for commutative semigroups},
Proc.~Amer.~Math.~Soc.\ 19 (1968), 1003.

\bibitem{compasympdelta}
J. Garc\'ia-Garc\'ia, M. Moreno-Fr\'ias, and A. Vigneron-Tenorio, 
\emph{Computation of Delta sets of numerical monoids}, 
Monatshefte f\"ur Mathematik 178 (3) 457--472.  

\bibitem{compasympomega}
J. Garc\'ia-Garc\'ia, M. Moreno-Fr\'ias, and A. Vigneron-Tenorio, 
\emph{Computation of the $\omega$-primality and asymptotic $\omega$-primality with applications to numerical semigroups}, 
Israel J. Math. 206 (2015), no. 1, 395--411.

\bibitem{deltadim3} 
P.~Garc\'ia-S\'anchez, D.~Llena, and A.~Moscariello, 
\emph{Delta sets for numerical semigroups with embedding dimension three}, 
preprint.  Available at \textsf{arXiv: math.AC/1504.02116}

\bibitem{factorasymp} 
F. Halter-Koch, 
\emph{On the asymptotic behavior of the number of distinct factorizations into irreducibles}, 
Arkiv f\"or Matematik, 31 \textbf{2} (1993) 297--305.  




\bibitem{factorhilbert}
C. O'Neill, 
\emph{On factorization invariants and Hilbert functions},
preprint.  Available at \textsf{arXiv: math.AC/1503.08351}.

\bibitem{omegaquasi}
C. O'Neill and R. Pelayo, 
\emph{On the Linearity of $\omega$-primality in Numerical Monoids},
J. Pure and Applied Algebra \textbf{218} (2014) 1620--1627.  

\bibitem{omegamonthly}
C. O'Neill and R. Pelayo, 
\emph{How do you measure primality?},
American Mathematical Monthly \textbf{122} (2015), no.~2, 121--137.  

\bibitem{pseudofrob}
J.~Rosales and M.~Branco,
\emph{Decomposition of a numerical semigroup as an intersection of irreducible numerical semigroups}, 
B.~Belg.~Math.~Soc-Sim. 9 (2002), 373--381.

\bibitem{numerical}
J.~Rosales and P.~Garc\'ia-S\'anchez, 
\emph{Numerical semigroups}, 
Developments in Mathematics, Vol. 20, Springer-Verlag, New York, 2009.





\end{thebibliography}
\end{document}